\documentclass[11pt]{amsart}

\usepackage[authoryear]{natbib}

\RequirePackage[OT1]{fontenc}
\RequirePackage{amsthm,amsmath,amssymb}
\usepackage[colorlinks=true, urlcolor=blue, linkcolor=black, citecolor=black, pdftex]{hyperref}
\usepackage{graphicx}
\usepackage[usenames]{color}

\usepackage{epstopdf} 

\def\1{1\hskip-2.6pt{\rm l}}

\def\R{{\mathbb{R}}}

\def\E{{\mathbb{E}}}
\def\P{{\mathbb{P}}}

\def\A{{\mathcal A}}

\def\M{\mathcal M}
\def\NN{{{\mathcal N}}}

\def\X{{\mathcal{X}}}
\def\S{{\mathcal S}_{n}}
\def\T{{\mathcal T}}

\def\eps{{\xi}}

\newcommand{\pen}{\mathop{\rm pen}\nolimits}
\newcommand{\eref}[1]{(\ref{#1})}
\newcommand{\pa}[1]{\left({#1}\right)}
\newcommand{\cro}[1]{\left[{#1}\right]}
\newcommand{\ab}[1]{\left|{#1}\right|}
\newcommand{\ac}[1]{\left\{{#1}\right\}}

\newtheorem{thm}{Theorem}
\newtheorem{lem}{Lemma}
\newtheorem{prop}{Proposition}
\newtheorem{cor}{Corollary}

\newtheorem{Ass}{Assumption}

\def\telque{\big |}

\def\<{{\langle}}
\def\>{{\rangle}}

\usepackage{memhfixc}  

\begin{document}
\title[Bernstein-type inequality]{A Bernstein-type inequality for suprema of random processes with an application to statistics}
\date{November, 2008} 
\author{Yannick Baraud} 
\address{Universit\'e de Nice Sophia-Antipolis, Laboratoire J-A Dieudonn\'e,
  Parc Valrose, 06108 Nice cedex 02} 
\email{baraud@math.unice.fr}
\begin{abstract}
We use the generic chaining device proposed by Talagrand to establish exponential bounds on the deviation probability of some suprema of random processes. Then, given a random vector $\xi$ in $\R^{n}$ the components of which are independent and admit a suitable exponential moment, we deduce a deviation inequality for the  squared Euclidean norm of the projection of $\xi$ onto a linear subspace of $\R^{n}$. Finally, we provide an application of such an inequality to statistics, performing model selection in the regression setting when the errors are possibly non-Gaussian and the collection of models possibly large.  
\end{abstract}
\maketitle
%

\section{introduction}

\subsection{Controlling suprema of random processes}
Let $\pa{X_{t}}_{t\in T}$ be real-valued and centered random variables indexed by a countable and nonempty set $T$ and 
\[
Z=\sup_{t\in T}X_{t}.
\]   
A central problem in Probability and Statistics is to provide a suitable control of the probability of deviation of $Z$. When $T$ is a (countable) bounded subset of a metric space $(\X,d)$, a common technique is to use a {\it chaining device}. The basic idea is to decompose $X_{t}$ into series of the form 
\[
X_{t}=\sum_{k\ge 0} X_{t_{k+1}}-X_{t_{k}}
\]
where $X_{t_{0}}=0$  a.s. and the $(t_{k})_{k\ge 1}$ is sequence of elements of $T$ converging towards $t$ and such that for each $k$, $t_{k}$ belongs to a suitable finite subset $T_{k}$ of $T$. Then, the control of $\sup_{t\in T} X_{t}$ amounts to those of the increments $X_{t_{k+1}}-X_{t_{k}}$ simultaneously for all $k$ and all pairs of elements $(t_{k},t_{k+1})\in T_{k}\times T_{k+1}$ which are close. This approach seems to go back to Kolmogorov and was very popular in Statistics in the 90s to control suprema of empirical processes with regard to the entropy of $T$, see van de Geer~\citeyearpar{Geer90} and Barron {\it et al}~\citeyearpar{MR1679028} for example. However, this approach suffers from the drawback that it leads to pessimistic numerical constants that are in general too large to be used in statistical procedures. An alternative to chaining is the use of the concentration phenomenon of some probability measures such as the Gaussian distribution for instance. Indeed, when the $X_{t}$ are Gaussian, for all $u\ge 0$ we have 
\begin{equation}\label{gaussien}
\P\pa{Z\ge \E\pa{Z}+\sqrt{2v u}}\le e^{-u}\ \ \ {\rm where}\ \ \ v=\sup_{t\in T}{\rm var}(X_{t}).
\end{equation}
This inequality is due to Sudakov \& Cirel'son~\citeyearpar{MR0365680}. A nice features of~\eref{gaussien} lies in the fact that it allows to recover the usual deviation bound for Gaussian random variables when $T$ reduces to a single element. Compared to chaining, Inequality~\eref{gaussien} provides a powerful tool for controlling suprema of Gaussian processes as soon as one is able to  evaluate $\E(Z)$ sharply enough.  

It is the merit of  Talagrand~\citeyearpar{MR1361756} to extend this approach for the purpose of controlling suprema of empirical processes, that is, when 
$X_{t}$ takes the form $\sum_{i=1}^{n}t(\xi_{i})-\E\pa{t(\xi_{i})}$ with $T$  a set of uniformly bounded functions and $\xi_{i}$ independent random variables. Yet, the original result by Talagrand involved suboptimal numerical constants and many efforts were made to recover it with sharper ones. A first step in this direction is due to Ledoux~\citeyearpar{MR1399224} by mean of nice entropy and tensorisation arguments. Then, further refinements were made on Ledoux's result by Massart~\citeyearpar{MR1782276}, Rio~\citeyearpar{MR1955352} and Bousquet~\citeyearpar{MR1890640}, the latter author achieving the best possible result in terms of constants. Nowadays, these entropy arguments have become a popular way of establishing deviation and concentration inequalities for $Z$ around its expectation. For a nice and complete introduction to these inequalities (and their applications to statistics) we refer the reader to the book by Massart~\citeyearpar{MR2319879}. 

Bousquet's inequality can be recovered (with worse constants) by applying the following result of Klein \& Rio~\citeyearpar{MR2135312} (Theorem~1.1). Actually, we write it in a slightly different form with possibly larger constants.  
\begin{thm}[Klein \& Rio]\label{KR}
For each $t\in T$,  let $\pa{\overline{X}_{i,t}}_{i=1,\ldots,n}$ be independent (but not necessarily i.i.d.) centered random variables with values in $[-c,c]$ and set $X_{t}=\sum_{i=1}^{n}\overline{X}_{i,t}$. For all $u\ge 0$, 
\begin{equation}\label{bern0}
\P\pa{Z\ge \E(Z)+ \sqrt{\pa{2v^{2}+2c\E(Z)}u}+3cu}\le \exp\pa{-u}
\end{equation}
where $v^{2}=\sup_{t\in T}{\rm var}\pa{X_{t}}$.
\end{thm}

This inequality should be compared to Bernstein's inequality that we recall below (see also Massart~\citeyearpar{MR2319879} for related conditions). Indeed, it can be shown that a sum $X$ of independent centered random variables $X_{i}=\overline{X}_{i}$ with values in $[-c,c]$ for $i=1,\ldots,n$ do satisfy the Condition~\eref{bern-moment} below with $v^{2}={\rm var}\pa{X}$. Consequently, Inequality~\eref{bern0} generalizes Bernstein's (with worse constants) to suprema of countable families of such $X$. 

\begin{thm}[Bernstein's inequality]
Let $X_{1},\ldots,X_{n}$ be independent random variables and set $X=\sum_{i=1}^{n }\pa{X_{i}-\E(X_{i})}$. Assume that there exist nonnegative numbers  $v,c$ such that for all $k\ge 3$
\begin{equation}\label{bern-moment}
\sum_{i=1}^{n}\E\cro{\ab{X_{i}}^{k}}\le {k!\over 2}v^{2}c^{k-2}
\end{equation}
Then,  for all $u\ge 0$
\begin{equation}\label{bern}
\P\pa{X \ge \sqrt{2v^{2}u}+cu}\le e^{-u}.
\end{equation}
Besides, for all $x\ge 0$, 
\begin{equation}\label{bern01}
\P\pa{X\ge x}\le \exp\pa{-{x^{2}\over 2(v^{2}+cx)}}.
\end{equation}
\end{thm}
In the literature, \eref{bern-moment} together with the fact that the $X_{i}$ are independent is sometime replaced by the weaker condition 
\begin{equation}\label{exp}
\E\pa{e^{\lambda X}}\le \exp\cro{{\lambda^{2}v^{2}\over 2(1-\lambda c)}},\ \ \ \ \forall \lambda\in (0,c).
\end{equation}
In this paper, we shall mainly deal with this type of assumption which has the advantage to depend on the law of $X$ only. 

Looking at condition~\eref{exp}, a natural question arises. Is it possible to establish an analogue of Klein \& Rio's result when one replaces the assumption that the ${\overline{X}_{i,t}}$ belong to $[-c,c]$ by a suitable assumption on $T$ and the Laplace transforms of the $X_{t}$? An attempt at solving this problem can be found in Bousquet~\citeyearpar{MR2073435}. There, the author considered the case $X_{t}=\sum_{i=1}^{n}\xi_{i}t_{i}$ where the $T$ is a subset of $[-1,1]^{n}$ and the $\xi_{i}$ independent and centered random variables satisfying
\begin{equation}\label{moment}
\E\cro{\ab{\xi_{i}}^{k}}\le {k!\over 2}\sigma^{2}c^{k-2},\ \ \forall \ k\ge 2
\end{equation}
which implies~\eref{exp} with $v^{2}=v^{2}(t)=\ab{t}_{2}^{2}\sigma^{2}$. Unfortunately, it turns that the result by Bousquet provides an analogue of~\eref{bern0} with $v^{2}$ replaced by $n\sigma^{2}$ although one would expect the smaller quantity $v^{2}=\sup_{t\in T}v^{2}(t)$. 

\subsection{Chi-square type random variables and model selection}
Originally, this result by Bousquet above was motivated by a statistical application. In order to give an account of how such processes arise in Statistics, consider the problem of estimating $f$ from the  observation of the random vector $Y=f+\xi$ in $\R^{n}$. Given a linear subspace $S$ of $\R^{n}$, the classical least-squares estimator of $f$ in $S$ is given by $\hat f=\Pi_{S}Y=\Pi_{S}f+\Pi_{S}\xi$ where $\Pi_{S}$ denotes the orthogonal projector onto $S$. Since the Euclidean (squared) distance beween $f$ and $\hat f$ decomposes as $\ab{f-\hat f}_{2}^{2}=\ab{f-\Pi_{S}f}_{2}^{2}+\ab{\Pi_{S}\xi}_{2}^{2}$, the study of the quadratic loss $\ab{f-\hat f}_{2}^{2}$ requires that of its random component $\ab{\Pi_{S}\xi}_{2}^{2}$. This quantity is usually called a $\chi^{2}$-type variable by analogy to the Gaussian case. Its study is connected to that of $Z$ by the formula
\[
\ab{\Pi_{S}\xi}_{2}=\sup_{t\in T}\sum_{i=1}^{n}\xi_{i}t_{i}=Z,
\]
where $T$ is countable and dense subset of the (Euclidean) unit ball of $S$. The control of such random variables is fundamental to perform model selection from the observation of $Y$ in the regression setting. When the $\xi_{i}$ admit few finite moments only, a control of such a $Z$ can be found in Baraud~\citeyearpar{MR1777129} by mean of a Rosenthal's type inequality.  By using chaining techniques, Baraud, Comte \& Viennet~\citeyearpar{MR1845321} handled the case of sub-Gaussian  $\xi_{i}$. The Gaussian case was studied by Birg\'e \& Massart~\citeyearpar{MR1848946} by using the concentration Inequality~\eref{gaussien}. More recently, Sauv\'e~\citeyearpar{MSauve} considered $\xi_{i}$ which satisfy~\eref{moment}. She discussed the fact that the inequality obtained in Bousquet~\citeyearpar{MR2073435} was unfortunately inadequate for controlling $\ab{\Pi_{S}\xi}_{2}^{2}$ and she solved the problem when $S$  consists of vectors the components of which are constant on each element of a given partition. 

\subsection{What is this paper about?}
In this paper, our motivations are twofold. First, we present an exponential bound for the probability of deviation of $Z=\sup_{i\in T}X_{t}$ under a suitable bound on the Laplace transform of the increments $X_{t}-X_{s}$ with $s,t\in T$. Our approach is inspired by that described in the book of Talagrand~\citeyearpar{MR2133757} for evaluating the expectations of suprema of random variables.   
Talagrand's approach relies on the idea of decomposing $T$ into partitions rather than into nets as it was usually done before. By using such a technique, the inequalities we get suffer from the usual drawback that the numerical constants are non-optimal but at least they allow a suitable control of $\chi^{2}$-type random variables over more general linear spaces $S$ than those considered in Sauv\'e~\citeyearpar{MSauve}. Second, we shall apply these inequalities for the purpose of selecting an appropriate least-squares estimator among a (possibly exponentially large) collection of candidate ones. If one excepts the case of histogram-type estimators, it seems that performing model selection in this context under the assumption that the errors satisfy~\eref{moment} is new. Besides, unlike Sauv\'e~\citeyearpar{MSauve}, our estimation procedure does not assume that an upper bound for the sup-nom of the regression function is known. 

The paper is organized as follows. We present our deviation bound for $Z$ in Section~\ref{sect:I}. We give an application to Statistics in Section~\ref{sect:stats}. We perform  model selection for the purpose of estimating the mean of a random vector. We shall restrict there to collections of models based on linear spans of piecewise or trigonometric polynomials. The case of more general linear spaces will be considered in Section~\ref{sec:unifions}. Section~\ref{Proof} is devoted to the proofs. 

Along the paper we shall assume that $n\ge 2$ and use the following notations. We denote by $e_{1},\ldots,e_{n}$ the canonical basis of $\R^{n}$ which we endow with the Euclidean inner product denoted $\<.,.\>$. For $x\in\R^{n}$, we set
\[
|x|_{2}=\sqrt{\<x,x\>},\ \  |x|_{1}=\sum_{i=1}^{n}|x_{i}|\ \ {\rm and}\ \ |x|_{\infty}=\max_{i=1,\ldots,n}|x_{i}|.
\]
The linear span of a family $u_{1},\ldots,u_{k}$ of vectors is denoted by ${\rm Span\!}\ac{u_{1},\ldots,u_{k}}$. The quantity $|I|$ is the cardinality of a finite set $I$.  Finally, $\kappa$ denotes the numerical constant $18$. It appears in the control of the deviation of $Z$ when applying Talagrand's chaining argument. As a consequence, it will appear all along the paper and it seems to us interesting to stress up how this constant is involved in the statistical procedure we propose. 

\section{A Talagrand-type Chaining argument for controlling suprema of random variables}\label{sect:I}

Let $\pa{X_{t}}_{t\in T}$ be a family of real valued and centered random variables indexed by a countable and nonempty set $T$.  Fix some $t_{0}$ in $T$ and set
\[
Z=\sup_{t\in T}\pa{X_{t}-X_{t_{0}}}\ \ \ {\rm and}\ \ \ \overline{Z}=\sup_{t\in T}\ab{X_{t}-X_{t_{0}}}.
\]
Our aim is to give a probabilistic control of the deviations of $Z$ (and $\overline Z$). We make the following assumptions
\begin{Ass}\label{momexp}
There exists two distances $d$ and $\delta$ on $T$ and a nonnegative constant $c$ such that for all $s,t\in T$ ($s\ne t$) 
\begin{equation}\label{debase}
\E\cro{e^{\lambda(X_{t}-X_{s})}}\le \exp\cro{{\lambda^{2}d^{2}(s,t)\over 2(1-\lambda c\delta(s,t))}},\ \ \forall \lambda\in\left[0, {1\over c\delta(s,t)}\right)
\end{equation}
with the convention $1/0=+\infty$.
\end{Ass}
The case $c=0$ corresponds to the situation where the increments of the process $X_{t}$ are {\it sub-Gaussian}. 

In this section, we also assume that $d$ and $\delta$ derive from norms. This is the only case we need to consider to handle the statistical problem described in Section~\ref{sect:stats}. Nevertheless, a more general result with arbitrary distances can be found in Section~\ref{Proof}. 

\begin{Ass}\label{lineaire}
Let $S$ be a linear space $S$ with dimension $D<+\infty$ endowed with two arbitrary norms denoted $\|\ \|_{2}$ and $\|\ \|_{\infty}$ respectively. The set $T$ is a subset of $S$ and for all $s,t\in T$, $d(s,t)=\|t-s \|_{2}$ and $\delta(s,t)=\|s-t \|_{\infty}$. Besides, 
\[
T\subset \ac{t\in S\ \telque\ \|t-t_{0}\|_{2}\le v,\ \ c\|t-t_{0}\|_{\infty}\le b}.
\]
\end{Ass}

Then, the following result holds. 
\begin{thm}\label{norm}
Under Assumptions~\ref{momexp} and~\ref{lineaire}, 
\begin{equation}\label{svanorm}
\P\cro{Z\ge \kappa\pa{\sqrt{v^{2}(D+x)}+b(D+x)}}\le e^{-x},\ \ \forall x\ge 0
\end{equation}
with $\kappa=18$. Moreover 
\begin{equation}\label{vanorm}
\P\cro{\overline{Z}\ge \kappa\pa{\sqrt{v^{2}(D+x)}+b(D+x)}}\le 2e^{-x},\ \ \forall x\ge 0.
\end{equation}
\end{thm}

If $T$ is no longer countable but admits a countable dense subset $T'$ (with respect to $\|\ \|_{2}$ or $\|\ \|_{\infty}$, both norms being equivalent on $S$) and if  the paths $t\mapsto X_{t}$ are continuous with probability 1, Theorem~\ref{norm} still holds since 
\[
\sup_{t\in T}\pa{X_{t}-X_{t_{0}}}=\sup_{t\in T'}\pa{X_{t}-X_{t_{0}}} \ \ \ a.s..
\]
Let us now turn to some examples. In the sequel, we take $t_{0}=0$, $T\subset \R^{n}$ and $X_{t}=\<\xi,t\>$ where the random vector $\xi=(\xi_{1},\ldots,\xi_{n})$ has independent and centered components. 

\subsubsection*{Comparison with the (sub)Gaussian case } 
Assume that for some $a>0$
\begin{equation}\label{subg}
\max_{i=1,\ldots,n}\log\E\cro{e^{\lambda \xi_{i}}}\le {\lambda^{2}a^{2}\over 2},\ \ \ \ \forall \lambda\in \R.
\end{equation}
This assumption holds when the $\xi_{i}$ are Gausian with mean 0 and variance $a^{2}$ or when the $\xi_{i}$ are bounded by $a$ for example. Consider some linear subspace $S$ of $\R^{n}$ with dimension $D$ and $T$ the Euclidean ball of $S$ centered at 0 of radius $r>0$. It follows from~\eref{subg} that Assumptions~\ref{momexp} and~\ref{lineaire} hold with $c=0$, $b=0$, $d(s,t)=\|t-s\|_{2}=a\ab{t-s}_{2}$ and $v=a r$. On the one hand, we obtain from Theorem~\ref{norm} the inequality
\begin{equation}\label{b1}
\P\cro{Z\ge \kappa ar\pa{\sqrt{D}+\sqrt{x}}}\le \P\cro{Z\ge \kappa ar\sqrt{D+x}}\le e^{-x},\ \forall x\ge 0.
\end{equation}
In view of commenting this bound, let us compare it to Inequality~\eref{gaussien} when the $\xi_{i}$ are Gaussian. In this case,  $\sup_{t\in T}{\rm var}(X_{t})=a^{2}r^{2}$ and since $Z^{2}/(a r)^{2}$ is a $\chi^{2}$ random variables with $D$ degrees of freedom, $\E(Z)\le \E^{1/2}(Z^{2})\le ar\sqrt{D}$. Hence, Inequality~\eref{gaussien} give,  on the other hand,
\[
\P\cro{Z\ge ar\pa{\sqrt{D}+\sqrt{x}}}\le e^{-x}.
\]
Except for the numerical constant $\kappa$, we see that this bound is comparable to~\eref{b1}. One could argue that the original bound~\eref{gaussien} is better since we have replaced $\E(Z)$ by the upper bound $ar\sqrt{D}$ but in fact, it can easily be checked that this quantity gives the right order of magnitude of $\E(Z)$ since $\E(Z)\ge ar\sqrt{2\pi^{-1}D}$.

\subsubsection*{Comparison with Inequalities~\eref{bern} and~\eref{KR}}
Assume now that $\xi$ satisfies for some positive numbers $\sigma$ and $c$, 
\begin{equation}\label{subg2}
\max_{i=1,\ldots,n}\log\E\cro{e^{\lambda \xi_{i}}}\le {\lambda^{2}\sigma^{2}\over 2(1-|\lambda| c)},\ \ \ \ \forall \lambda\in (-1/c,1/c).
\end{equation}
As a first simple example, let us take $S={\rm Span\!}\ac{\1}$ where $\1=(1,\ldots,1)'\in\R^{n}$ and $T=\ac{\lambda\1,\ \lambda\in[-1,1]}$. Under~\eref{subg2}, Assumptions~\ref{momexp} and~\ref{lineaire} hold with $d(s,t)=\|s-t\|_{2}=\sigma\ab{t-s}_{2}$, $\delta(s,t)=\|s-t\|_{\infty}=\ab{s-t}_{\infty}=\max_{i=1,\ldots,n}\ab{s_{i}-t_{i}}$,
$v^{2}=n$ and $b=c$. We can therefore apply Theorem~\ref{norm} and get,
\begin{equation}\label{eq4}
\P\cro{Z\ge \kappa\pa{\sqrt{n(1+x)\sigma^{2}}+c(1+x)}}\le e^{-x},\ \ \forall x\ge 0.
\end{equation}
On the other hand, for such a set $T$, $Z$ is merely $\ab{\<\xi,\1\>}=\ab{\sum_{i=1}^{n}\xi_{i}}$ and by using Bernstein's Inequality~\eref{bern} twice (with $\xi$ and $-\xi$) and $u=x+\log(2)$, we derive
\[
\P\cro{Z\ge \sqrt{n(\log(2)+x)\sigma^{2}}+c(\log(2)+x)}\le e^{-x},\ \ \forall x\ge 0.
\]
This bound is comparable to~\eref{eq4}.

Let us now take $S$ as any linear subspace of $\R^{n}$ of dimension $D$,
\[
T= \ac{t\in S\ \telque\ \|t\|_{2}\le v,\ \ c\|t\|_{\infty}\le 1}
\]
and assume $\sigma=1$ for simplicity. When $|\xi_{i}|\le c$ for all $i$, we can compare our Inequality~\eref{svanorm} to that of Klein \& Rio (Inequality~\eref{bern0}) since the assumptions of Theorem~\ref{KR} and~\ref{norm} are both satisfied. On the one hand, the inequality by Klein \& Rio gives that with probability at least $1-e^{-x}$,  $Z\le z(x)$ where
\[
z(x)= \E(Z)+ \sqrt{\pa{2v^{2}+2c\E(Z)}x}+2cx.
\]
The concavity of $\log$ together with the elementary inequality $2ab\le a^{2}+b^{2}$ lead to the following upper and lower bounds for $z(x)$
\begin{eqnarray*}
\E(Z)+ \sqrt{2v^{2}x}+cx\le &z(x)&\le 3\pa{\E(Z)+ \sqrt{2v^{2}x}+cx}
\end{eqnarray*}
On the other hand, our inequality gives that with probability at least $1-e^{-x}$, $Z\le \kappa w(x)$ where
\[
w(x)=\sqrt{v^{2}(D+x)}+c(D+x)
\]
and similar computations yield
\begin{eqnarray*}
{1\over 2}\pa{\sqrt{Dv^{2}}+cD+\sqrt{v^{2}x}+cx}\le  &w(x)&\le \sqrt{Dv^{2}}+cD+\sqrt{v^{2}x}+cx.
\end{eqnarray*}
Except for the numerical constants, we see that the main difference between Klein \& Rio's Inequality and ours essentially lies in the fact that $\E(Z)$ is replaced by $E=\sqrt{Dv^{2}}+cD$. It follows from Cauchy-Schwarz's Inequality that 
\[
\E(Z)\le \sqrt{Dv^{2}}<E=\sqrt{Dv^{2}}+cD,
\]
showing that our bound $w(x)$ involves an upper bound for $\E(Z)$. Under the only assumption that $\xi$ satisfy~\eref{subg2}, the problem of replacing $E$ by $\E(Z)$ remains open. Nevertheless, the term $\sqrt{Dv^{2}}$ turns to be of order $\E(Z)$ in typical situations (think of the Gaussian case) and our bound becomes then comparable 
 to that given by Klein \& Rio as soon as $c^{2}D\le v^{2}$. This turns to be enough to derive deviations bounds for $\chi^{2}$-type random variables in many situations of interest as we shall see in Section~\ref{sect-chi}.

\section{An application to model selection in the regression framework}\label{sect:stats}
Let $Y$ be a random vector of $\R^{n}$ with independent components. In this section, our aim is to estimate $f=\E(Y)$ under the assumption that the components of the noise $\eps=Y-f$ satisfy 
\begin{equation}\label{bernstein}
\log\E\cro{e^{\lambda \eps_{i}}}\le {\lambda^{2}\sigma^{2}\over 2(1-|\lambda| c)},\ \ \forall \lambda\in(-1/c,1/c),\ \ i=1,\ldots,n
\end{equation}
for some known positive numbers $\sigma$ and $c$. 
Inequality~\eref{bernstein} holds for a large class of distributions (once suitably centered) including Poisson, exponential, Gamma... Besides,~\eref{bernstein} is fulfilled when the $\xi_{i}$ satisfy~\eref{moment}.

Our estimation strategy is based on model selection. We start with a (possibly large) collection $\ac{S_{m},\ m\in\M}$ of linear subspaces (models) of $\R^{n}$ and associate to each of these the least-squares estimators $\hat f_{m}=\Pi_{S_{m}}Y$. Given a penalty function $\pen$ from $\M$ to $\R_{+}$, we define the penalized criterion ${\rm crit}(.)$ on  $\M$ by
\begin{equation}\label{crit}
{\rm crit}(m)=\ab{Y-\hat f_{m}}_{2}^{2}+\pen(m).
\end{equation}
In this section, we propose to establish risk bounds for the estimator of $f$ given by $\hat f_{\hat m}$ where the index $\hat m$ is selected from the data among $\M$ as any minimizer of  ${\rm crit}(.)$. 

In the sequel, the penalty $\pen$ will be based on some {\it a priori} choice of nonnegative numbers $\ac{\Delta_{m},\ m\in\M}$ for which we set 
\[
\Sigma=\sum_{m\in\M} e^{-\Delta_{m}}<+\infty.
\]
When $\Sigma=1$, the choice of the $\Delta_{m}$ can be viewed as that of a prior distribution on the models $S_{m}$. For related conditions and their interpretation, see Barron and Cover~\citeyearpar{MR1111806} or  Barron {\it et al}~\citeyearpar{MR1679028}.

In the following sections, we give an account of our main result (to be presented in Section~\ref{sect:main}) for some typical collections of linear spaces $\ac{S_{m},\ m\in\M}$.

\subsection{Selecting among histogram-type estimators}
For a partition $m$ of $\ac{1,\ldots,n}$, $S_{m}$ denotes  the linear span of vectors of $\R^{n}$ the coordinates of which are constants on each element $I$ of $m$. In the sequel, we shall restrict to partitions $m$ the elements of which consist of consecutive integers. 

Consider a partition  $\mathfrak{m}$ of $\ac{1,\ldots,n}$ and $\M$ a collection of partitions $m$ such that $S_{m}\subset S_{\mathfrak{m}}$. We obtain the following result.

\begin{prop}\label{histo}
Let $a,b>0$. Assume that
\begin{equation}\label{condhisto}
|I|\ge a^{2}\log^{2}(n),\ \ \forall I\in \mathfrak{m}.
\end{equation}
If for some $K>1$, 
\begin{equation}\label{p1}
\pen(m)\ge K\kappa^{2}\pa{\sigma^{2}+2c{(\sigma+c)(b+2)\over a\kappa}}\pa{|m|+\Delta_{m}},\ \ \forall m\in\M.
\end{equation}
the estimator $\hat f_{\hat m}$ satisfies 
\begin{equation}\label{inehisto}
\E\pa{\ab{f-\hat f_{\hat m}}_{2}^{2}}\le C(K)\cro{\inf_{m\in\M}\cro{\E\pa{\ab{f-\hat f_{m}}_{2}^{2}}+\pen(m)}+R}
\end{equation}
where  $C(K)$ is given by~\eref{CK} and
\begin{equation*}\label{R(c,L)}
R=\kappa^{2}\pa{\sigma^{2}+2c{(c+\sigma)(b+2)\over a\kappa}}\Sigma+2{(c+\sigma)^{2}(b+2)^{2}\over a^{2}n^{b}}.
\end{equation*}
\end{prop}

Note that when $c=0$, Inequality~\eref{p1} holds as soon as 
\begin{equation}\label{penideal}
\pen(m)= K\kappa^{2}\sigma^{2}\pa{|m|+\Delta_{m}},\ \ \forall m\in\M.
\end{equation}
Besides, by taking $a=\log^{-1}(n)$ we see that Condition~\eref{condhisto} becomes automatically satisfied and by letting $b$ tend to $+\infty$, Inequality~\eref{inehisto} holds with $\pen$ given by~\eref{penideal} and $R=\kappa^{2}\sigma^{2}\Sigma$.

The problem of selecting among histogram-type estimators in this regression setting has recently been investigated in Sauv\'e~\citeyearpar{MSauve}. Her selection procedure is similar to ours with a different choice of the penalty term. Unlike hers, our penalty does not involve an upper bound $M$ (assumed to be known) on $\ab{f}_{\infty}$. 

\subsection{Families of piecewise polynomials}
In this section, we assume that $f$ is of the form $(F(1/n),\ldots, F(n/n))$ where $F$ is an unknown function on $(0,1]$. Our aim is to estimate $F$ by an estimator which is a piecewise polynomial of degree not larger than $d$ based on a data-driven choice of a partition of $(0,1]$. 

In the sequel, we shall consider partitions $m$ of $\ac{1,\ldots,n}$ such that each element  $I\in m$ consists of at least $d+1$ consecutive integers. For such a partition, $S_{m}$ denotes the linear span of vectors of the form $(P(1/n),\ldots,P(n/n))$ where $P$ varies among the space of piecewise polynomials with degree not larger than $d$ based on the partition of $(0,1]$ given by
\[
\ac{\left({\min I-1\over n}, {\max I\over n}\right],\ I\in m}.
\]
Consider a partition  $\mathfrak{m}$ of $\ac{1,\ldots,n}$ and $\M$ a collection of partitions $m$ such that $S_{m}\subset S_{\mathfrak{m}}$. We obtain the following result.

\begin{prop}\label{pp}
Let $a,b>0$. Assume that
\begin{equation}\label{condppm}
|I|\ge (d+1)a^{2}\log^{2}(n)\ge d+1,\ \ \ \forall I\in \mathfrak{m}.
\end{equation}
If for some $K>1$, 
\[
\pen(m)\ge K\kappa^{2}\pa{\sigma^{2}+c{4\sqrt{2}(\sigma+c)(d+1)(b+2)\over a\kappa}}\pa{D_{m}+\Delta_{m}},\ \ \forall m\in\M.
\] 
the estimator $\hat f_{\hat m}$ satisfies~\eref{inehisto} with 
\[
R=\kappa^{2}\pa{\sigma^{2}+c{4\sqrt{2}(\sigma+c)(d+1)(b+2)\over a\kappa}}\Sigma+4{(c+\sigma)^{2}(b+2)^{2}\over a^{2}n^{b}}.
\]
\end{prop}

\subsection{Families of trigonometric polynomials}
As in the previous section, we assume here that  $f$ is of the form $(F(x_{1}),\ldots, F(x_{n}))$ where $x_{i}=i/n$ for $i=1,\ldots,n$ and $F$ is an unknown function on $(0,1]$. Our aim is to estimate $F$ by a trigonometric polynomial of degree not larger than some $\overline D\ge 0$. 

Consider the (discrete) trigonometric system $\ac{\phi_{j}}_{j\ge 0}$  of vectors in $\R^{n}$ defined by
\begin{eqnarray*}
\phi_{0}&=&(1/\sqrt{n},\ldots,1/\sqrt{n})\\
\phi_{2j-1}&=&\sqrt{2\over n}\pa{\cos\pa{2\pi jx_{1}},\ldots,\cos\pa{2\pi jx_{1}}},\ \forall j\ge 1\\
\phi_{2j}&=&\sqrt{2\over n}\pa{\sin\pa{2\pi jx_{1}},\ldots,\sin\pa{2\pi jx_{1}}},\ \forall j\ge 1.
\end{eqnarray*}
Let $\M$ be a family of subsets of $\ac{0,\ldots,2\overline D}$. For $m\in\M$,  we define $S_{m}$ as the linear span of the $\phi_{j}$ with $j\in m$ (with the convention $S_{m}=\ac{0}$ when $m=\varnothing$). 

\begin{prop}\label{trigo}
Let $a,b>0$. Assume that $2\overline D+1\le \sqrt{n}/(a\log(n))$.
If for some $K>1$,
\[
\pen(m)\ge K\kappa^{2}\pa{\sigma^{2}+{4c(c+\sigma)(b+2)\over a}}\pa{D_{m}+\Delta_{m}},\ \ \forall m\in\M
\]
then  $\hat f_{\hat m}$ satisfies~\eref{inehisto} with 
\[
R=\kappa^{2}\pa{\sigma^{2}+{4c(c+\sigma)(b+2)\over a}}\Sigma+{4(b+2)^{2}(c+\sigma)^{2}\over a^{2}(2\overline D +1)n^{b}}.
\]

\end{prop}

\section{Towards a more general result}\label{sec:unifions}
We consider the statistical framework presented in Section~\ref{sect:stats} and give a general result that allows to handle Propositions~\ref{histo}, ~
\ref{pp} and~\ref{trigo} simultaneously. It will rely on some geometric properties of the linear spaces $S_{m}$ that we describe below. 

\subsection{Some geometric quantities}
Let $S$ be a linear subspace of $\R^{n}$. We associate to $S$ the following quantities 
\begin{equation}\label{defL}
\Lambda_{2}(S)=\max_{i=1,\ldots,n}|\Pi_{S}e_{i}|_{2}\ \ {\rm and}\ \ \Lambda_{\infty}(S)=\max_{i=1,\ldots,n}|\Pi_{S}e_{i}|_{1}.
\end{equation}
It is not difficult to see that these quantities can be interpreted in terms of norm connexions, more precisely
\[
\Lambda_{2}(S)=\sup_{t\in S\setminus\ac{0}}{\ab{t}_{\infty}\over \ab{t}_{2}}\ \ {\rm and}\ \ \Lambda_{\infty}(S)=\sup_{t\in \R^{n}\setminus\ac{0}}{\ab{\Pi_{S}t}_{\infty}\over \ab{t}_{\infty}}.
\]
Clearly, $\Lambda_{2}(S)\le 1$. Besides, since $\ab{x}_{1}\le \sqrt{n}\ab{x}_{2}$ for all $x\in\R^{n}$, $\Lambda_{\infty}(S)\le \sqrt{n}\Lambda_{2}(S)$. Nevertheless, these bounds can be rather rough as shown by the following proposition.

\begin{prop}\label{control-lambda}
Let $P$ be some partition of $\ac{1,\ldots,n}$, $J$ some nonempty index set and 
\[
\ac{\phi_{j,I},\ (j,I)\in J\times P}
\]
an orthonormal system such that for some $\Phi>0$ and all $I\in P$
\[
\sup_{j\in J}\ab{\phi_{j,I}}_{\infty}\le {\Phi\over \sqrt{|I|}}\ \ {\rm and}\ \ \<\phi_{j,I},e_{i}\>=0\  \forall i\not \in I.
\]
If $S$ is the linear span of the $\phi_{j,I}$ with $(j,I)\in J\times P$, 
\[
\Lambda_{2}^{2}(S)\le \pa{{|J|\Phi^{2}\over \min_{I\in P}|I|}}\wedge 1\ \ {\rm and}\ \ \Lambda_{\infty}(S)\le \pa{|J|\Phi^{2}}\wedge \pa{\sqrt{n}\Lambda_{2}(S)}.
\]
\end{prop}

\begin{proof}[Proof of Proposition~\ref{control-lambda}]
We have already seen that $\Lambda_{2}(S)\le 1$ and $\Lambda_{\infty}(S)\le \sqrt{n}\Lambda_{2}(S)$, so it remains to show that
\[
\Lambda_{2}^{2}(S)\le {|J|\Phi^{2}\over \min_{I\in P}|I|}\ \ {\rm and}\ \ \Lambda_{\infty}(S)\le |J|\Phi^{2}.
\]
Let $i=1,\ldots,n$. There exists some unique $I\in P$ such that $i\in I$ and since $\<\phi_{j,I'},e_{i}\>=0$ for all $I'\neq I$, 
\[
\Pi_{S}e_{i}=\sum_{j\in J}\<e_{i},\phi_{j,I}\>\phi_{j,I}.
\]
Consequently, 
\[
\ab{\Pi_{S}e_{i}}_{2}^{2}=\sum_{j\in J}\<e_{i},\phi_{j,I}\>^{2}\le {|J|\Phi^{2}\over |I|}\le {|J|\Phi^{2}\over \min_{I\in P}|I|}
\]
and 
\begin{eqnarray*}
\ab{\Pi_{S}e_{i}}_{1}&=&\sum_{i'\in I}\ab{\sum_{j\in J}\<e_{i},\phi_{j,I}\>\<e_{i'},\phi_{j,I}\>}\le |I|{|J|\Phi^{2}\over |I|}\le |J|\Phi^{2}.
\end{eqnarray*}
We conclude since  $i$ is arbitrary. 
\end{proof}

\subsection{The main result}\label{sect:main}
Let $\ac{S_{m},\ m\in\M}$ be family of linear spaces and $\ac{\Delta_{m},\ m\in\M}$ a family of nonnegative weights. We define $\S=\sum_{m\in\M}S_{m}$ and
\[
\overline{\Lambda}_{\infty}=\pa{\sup_{m,m'\in\M}\Lambda_{\infty}(S_{m}+S_{m'})}\vee 1.
\]

\begin{thm}\label{selmod}
Let $K>1$ and $z\ge 0$. Assume that for all $i=1,\ldots,n$, Inequality~\eref{bernstein} holds.  Let $\pen$ be some penalty function satisfying 
\begin{equation}\label{pen}
\pen(m)\ge K\kappa^{2}\pa{\sigma^{2}+{2cu\over \kappa}}\pa{D_{m}+\Delta_{m}},\ \ \forall m\in\M
\end{equation}
where
\begin{equation}\label{defu}
u=(c+\sigma)\overline{\Lambda}_{\infty}\Lambda_{2}(\S)\log(n^{2}e^{z}).
\end{equation}
If one selects $\hat m$ among $\M$ as any minimizer of ${\rm crit}(.)$ defined by~\eref{crit} then 
\[
\E\cro{\ab{f-\hat f_{\hat m}}_{2}^{2}}\le C(K)\cro{\inf_{m\in\M}\pa{\E\cro{\ab{f-\hat f_{m}}_{2}^{2}}+\pen(m)}+R}
\]
where 
\begin{eqnarray}
C(K)&=& {K(K^{2}+K-1)\over (K-1)^{3}}\label{CK}
\end{eqnarray}
and 
\begin{eqnarray}
R&=&\kappa^{2}\pa{\sigma^{2}+{2cu\over \kappa}}\Sigma+2\pa{u\over \overline{\Lambda}_{\infty}}^{2}e^{-z}.\nonumber
\end{eqnarray}
\end{thm}

When $c=0$ we derive the following corollary by letting $z$ grow towards infinity. 

\begin{cor}
Let $K>1$. Assume that the $\eps_{i}$ for $i=1,\ldots,n$ satisfy Inequality~\eref{bernstein} with $c=0$. If one selects $\hat m$ among $\M$ as a minimizer of ${\rm crit}$ defined by~\eref{crit}
with $\pen$ satisfying
\[
\pen(m)\ge K\kappa^{2}\sigma^{2}\pa{D_{m}+\Delta_{m}},\ \ \forall m\in\M
\]
then 
\[
\E\cro{\ab{f-\hat f_{\hat m}}_{2}^{2}}\le {K(K^{2}+K-1)\over (K-1)^{3}}\inf_{m\in\M}\pa{\E\cro{\ab{f-\hat f_{m}}_{2}^{2}}+\pen(m)}+R
\]
where 
\[
R={K^{3}\kappa^{2}\sigma^{2}\over (K-1)^{2}}\Sigma.
\]
\end{cor}

\section{Proofs}\label{Proof}
We start with the following result generalizing
Theorem~\ref{norm} when $d$ and $\delta$ are not induced by norms. We assume that $T$ is finite and take numbers $v$ and $b$ such that 
\begin{equation}\label{contraintes}
\sup_{s\in T}d(s,t_{0})\le v,\ \ \ \sup_{s\in T}c\delta(s,t_{0})\le b.
\end{equation}
We consider now a family of finite partitions $\pa{\A_{k}}_{k\ge 0}$ of $T$, such that $\A_{0}=\ac{T}$ and for $k\ge 1$ and $A\in \A_{k}$
\[
d(s,t)\le 2^{-k}v\ \ {\rm and}\ \  c\delta(s,t)\le 2^{-k} b,\ \ \forall s,t\in A.
\]
Besides, we assume $\A_{k}\subset \A_{k-1}$ for all $k\ge 1$, which means that all elements $A\in \A_{k}$ are subsets of an element of $\A_{k-1}$. Finally, we define for $k\ge 0$
\[
N_{k}=|\A_{k+1}||\A_{k}|.
\]

\begin{thm}\label{chi2}
Let $T$ be some finite set. Under Assumption~\ref{momexp}, 
\begin{equation}\label{sva}
\P\pa{Z\ge H+2\sqrt{2v^{2}x}+ 2bx}\le e^{-x},\ \ \forall x>0
\end{equation}
where
\[
H=\sum_{k\ge 0}2^{-k}\pa{v\sqrt{2\log(2^{k+1}N_{k})}+b\log(2^{k+1}N_{k})}.
\]
Moreover, 
\begin{equation}\label{va}
\P\pa{\overline{Z}\ge H+2\sqrt{2v^{2}x}+ 2bx}\le 2e^{-x},\ \ \forall x>0.
\end{equation}
\end{thm}

The quantity $H$ can be related to the entropies of $T$ with respect to the distances $d$ and $c\delta$ (when $c\ne 0$) in the following way. We first recall that for a distance $e(.,.)$ on $T$ and $\varepsilon>0$, the entropy $H(T,e,\varepsilon)$ is defined as logarithm of the minimum number of balls of radius $\varepsilon$ with respect to $e$ which are necessary to cover $T$. Note that for $k\ge 0$, each element $A$ of the partition $\A_{k+1}$ is a subset of both a ball of radius $2^{-(k+1)}v$ with respect to $d$ and of a ball of radius $2^{-(k+1)}b$ with respect $c\delta$. Besides, since $|\A_{k+1}|\le N_{k}$, we obtain that for all $\varepsilon\in[2^{-(k+1)},2^{-k})$
\[
H(T,\varepsilon)=\max\ac{H(T,\varepsilon v),H(T,c\delta,\varepsilon b)}\le \log(N_{k}).
\]
By integrating with respect to $\varepsilon$ (and using~\eref{contraintes}), we deduce that
\[
\int_{0}^{+\infty}\pa{\sqrt{2v^{2}H(T,\varepsilon)}+bH(T, \varepsilon)}d\varepsilon\le H.
\]

\subsection{Proof of Theorem~\ref{chi2}}
Note that we obtain~\eref{va} by using~\eref{sva} twice (once with $X_{t}$ and then with $-X_{t}$). Let us now prove~\eref{sva}. For each $k\ge 1$ and $A\in\A_{k}$, we choose some arbitrary element $t_{k}(A)$ in $A$. For each $t\in T$ and $k\ge 1$, there exists a unique $A\in \A_{k}$ such that $t\in A$ and we set $\pi_{k}(t)=t_{k}(A)$. When $k=0$, we set $\pi_{0}(t)=t_{0}$.

We consider the (finite) decomposition
\[
X_{t}-X_{t_{0}}=\sum_{k\ge 0}X_{\pi_{k+1}(t)}-X_{\pi_{k}(t)}
\]
and set for $k\ge 0$
\[
z_{k}=2^{-k}\pa{v\sqrt{2\pa{\log(2^{k+1}N_{k})+x}}\ +\ b\pa{\log(2^{k+1}N_{k})+x}}
\]
Since $\sum_{k\ge 0}z_{k}\le z=H+2v\sqrt{2x}+ 2bx$,
\begin{eqnarray*}
\P\pa{Z \ge z} &\le& \P\pa{\exists t,\ \exists k\ge 0,\ \ X_{\pi_{k+1}(t)}-X_{\pi_{k}(t)}\ge  z_{k}}\\
&\le& \sum_{k\ge 0}\sum_{(s,u)\in E_{k}}\P\pa{X_{u}-X_{s}\ge z_{k}}
\end{eqnarray*}
where 
\[
E_{k}=\ac{\pa{\pi_{k}(t),\pi_{k+1}(t)}|\ t\in T}.
\]
Since $\A_{k+1}\subset \A_{k}$, $\pi_{k}(t)$ and $\pi_{k+1}(t)$ belong to a same element of $\A_{k}$ and therefore $d(s,u)\le 2^{-k}v$ and $c\delta(s,u)\le 2^{-k}b$ for all pairs $(s,u)\in E_{k}$. 
Besides, under Assumption~\ref{momexp}, the random variable $X=X_{u}-X_{s}$ with $(s,u)\in E_{k}$ is centered and satisfies~\eref{exp} with $2^{-k}v$ and $2^{-k}b$ in place of $v$ and $c$. Hence, by using Berstein's Inequality~\eref{bern}, we get for all $(s,u)\in E_{k}$ and $k\ge 0$
\[
\P\pa{X_{u}-X_{s}\ge z_{k}}\le 2^{-(k+1)}N_{k}^{-1}e^{-x}\le 2^{-(k+1)}|E_{k}|^{-1}e^{-x}.
\]
Finally, we obtain Inequality~\eref{sva} summing up this inequalities over $(s,u)\in E_{k}$ and $k\ge 0$. 

\subsection{Proof of Theorem~\ref{norm}}
We only prove~\eref{svanorm}, the argument for proving~\eref{vanorm} being the same as that for proving~\eref{va}. For $t\in S$ and $r>0$, we denote by $B_{2}(t,r)$ and $B_{\infty}(t,r)$ the balls centered at $t$ of radius $r$ associated to $\|\ \|_{2}$ and $\|\ \|_{\infty}$ respectively. In the sequel, we shall use the following result on the entropy of those balls. 

\begin{prop}\label{entropie}
Let $\|\ \|$ be an arbitrary norm on $S$ and $B(0,1)$ the corresponding unit ball. For each $\delta\in (0,1]$, the minimal number $\NN(\delta)$ of balls of radius $\delta$ (with respect to $\|\ \|$) which are necessary to cover $B(0,1)$ satisfies 
\[
\NN(\delta) \le \pa{1+2\delta^{-1}}^{D}.
\]
\end{prop}

This lemma can be found in Birg\'e~\citeyearpar{MR722129} (Lemma 4.5, p. 209) with a proof referring to Lorentz~\citeyearpar{MR0203320}. Nevertheless, we provide a proof below to keep this paper as self-contained as possible. 

\begin{proof}
With no loss of generality, we may assume that $S=\R^{D}$. Let $\delta\in (0,1]$. A subset $\T$ of $B(0,1)$ is called $\delta$-separated if for all $s,t\in\T$, $\|s-t\|>\delta$. If $\T$ is $\delta$-separated, the family of (open) balls centered at those $t\in\T$ with radius $\delta/2$ are all disjoint and included in the ball $B(0,1+\delta/2)$. By a volume argument (with respect to the Lebesgue measure on $\R^{D}$), we deduce that $\T$ is finite and satisfies  $|\T|\le (1+2\delta^{-1})^{D}$. Consider now a maximal $\delta$-separated set $\T$, that is  
\[
|\T|=\max_{\T'}|\T'|
\]
where $\T'$ runs among the family of all the $\delta$-separated subset of $B(0,1)$. By definition, for all $t\in B(0,1)\setminus \T$, $\T\cup\ac{t}$ is no longer a $\delta$-net and therefore that the family of balls $\ac{B(t,\delta),\ t\in\T}$ covers $B(0,1)$.  Consequently 
\[
\NN(\delta)\le |\T|\le (1+2\delta^{-1})^{D}.
\]
\end{proof}

Let us now turn to the proof of~\eref{svanorm}. Note that it is enough to prove that for some $u<H+2\sqrt{2v^{2}x}+ 2bx$ and all finite sets $T$ satisfying Inequalities~\eref{debase} and~\eref{contraintes}  
\[
\P\pa{\sup_{t\in T}\pa{X_{t}-X_{t_{0}}} > u}\le e^{-x}.
\]
Indeed, for any sequence $\pa{T_{n}}_{n\ge 0}$ of finite subsets of $T$ increasing towards $T$, that is, satisfying $T_{n}\subset T_{n+1}$ for all $n\ge 0$ and $\bigcup_{n\ge 0}T_{n}=T$, the sets 
\[
\ac{\sup_{t\in T_{n}}\pa{X_{t}-X_{t_{0}}} > u}
\]
increases (for the inclusion) towards $\ac{Z>u}$. Therefore, 
\[
\P\pa{Z> u}=\lim_{n\to +\infty}\P\pa{\sup_{t\in T_{n}}\pa{X_{t}-X_{t_{0}}} > u}.
\]
Consequently, we shall assume hereafter that $T$ is finite. 

For $k\ge 0$ and $j\in\ac{2,\infty}$ define the sets $\A_{j,k}$ as follows. We first consider the case $j=2$. For $k=0$, $\A_{2,0}=\ac{T}$. By applying Proposition~\ref{entropie} with $\|\ \|=\|\ \|_{2}/v$ and $\delta=1/4$, we can cover $T\subset B_{2}(t_{0},v)$ with at most $9^{D}$ balls with radius $v/4$. From such a finite covering $\ac{B_{1},\ldots,B_{N}}$ with $N\le 9^{D}$, it is easy to derive a partition $\A_{2,1}$ of $T$ by at most $9^{D}$ sets of diameter not larger than $v/2$. Indeed, $\A_{2,1}$ can merely consist of the non-empty sets among the family 
\[
\ac{\pa{B_{k}\setminus\bigcup_{1\le \ell<k}B_{\ell}}\cap T,\ \ k=1,\ldots,N}
\]
(with the convention $\bigcup_{\varnothing}=\varnothing$). Then, for $k\ge 2$, proceed by induction using Proposition~\ref{entropie} repeatedly. Each element $A\in\A_{2,k-1}$ is a subset of a ball of radius $2^{-k}v$ and can be partitioned similarly as before into $5^{D}$ subsets of balls of radii $2^{-(k+1)}v$. By doing so, the partitions $\A_{2,k}$ with $k\ge 1$ satisfy $\A_{2,k}\subset \A_{2,k-1}$, $|\A_{2,k}|\le (1.8)^{D}\times 5^{kD}$ and for all $A\in\A_{2,k}$, 
\[
\sup_{s,t\in A}\|s-t\|_{2}\le 2^{-k}v.
\]
Let us now turn to the case $j=+\infty$. If $c>0$, define the partitions $\A_{\infty, k}$ in exactly the same way as we did for the $\A_{2, k}$. Similarly, the partitions $\A_{\infty,k}$ with $k\ge 1$ satisfy $\A_{\infty,k}\subset \A_{\infty,k-1}$, $|\A_{\infty,k}|\le (1.8)^{D}\times 5^{kD}$ and for all $A\in\A_{\infty,k}$, 
\[
\sup_{s,t\in A}c\|s-t\|_{\infty}\le 2^{-k}b.
\]
When $c=0$, we simply take $\A_{\infty, k}=\ac{T}$ for all $k\ge 0$ and note that the properties above are fulfilled as well.

Finally, define the partition $\A_{k}$ for $k\ge 0$ as that generated by $\A_{2,k}$ and $\A_{\infty,k}$, that is 
\[
\A_{k}=\ac{A_{2}\cap A_{\infty}|\ A_{2}\in\A_{2,k},\ A_{\infty}\in\A_{\infty,k}}.
\]
Clearly, $\A_{k+1}\subset \A_{k}$. Besides, $|\A_{0}|=1$ and for $k\ge 1$, 
\[
|\A_{k}|\le |\A_{2,k}||\A_{\infty,k}|\le (1.8)^{2D}\times 5^{2kD}.
\]
The set $T$ being finite, we can apply Theorem~\ref{chi2}. Actually, our construction of the $\A_{k}$ allows us to slightly gain in the constants. Going back to the proof of Theorem~\ref{chi2}, we note that 
\[
|E_{k}|=|\ac{\pa{\pi_{k}(t),\pi_{k+1}(t)}|\ t\in T}|\le |\A_{k+1}|\le 9^{2D}\times 5^{2kD}
\]
since the element $\pi_{k+1}(t)$ determines $\pi_{k}(t)$ in a unique way. This means that one can take $N_{k}=9^{2D}\times 5^{2kD}$ in the proof of Theorem~\ref{chi2}. By taking the notations of Theorem~\ref{chi2}, we have,  
\begin{eqnarray*}
H&\le &\sum_{k\ge 0}2^{-k}\cro{v\sqrt{2\log(2^{k+1}\times 9^{2D}\times 5^{2kD})}+b\log\pa{2^{k+1}\times 9^{2D}\times 5^{2kD}}}\\
&<& 14\sqrt{Dv^{2}}+18Db
\end{eqnarray*}
and using the concavity of $x\mapsto \sqrt{x}$, we get 
\begin{eqnarray*}
H+2\sqrt{2v^{2}x}+ 2bx &\le & 14\sqrt{Dv^{2}}+2\sqrt{2v^{2}x}+18b(D+x)\\
&\le& 18\pa{\sqrt{v^{2}\pa{D+x}}+b(D+x)}.
\end{eqnarray*}
which leads to the result.

\subsection{A control of $\chi^{2}$-type random variables}\label{sect-chi}
We have the following result. 

\begin{thm}\label{chi}
Let $S$ be some linear subspace of $\R^{n}$ with dimension $D$. If the coordinates of $\eps$ are independent and satisfy~\eref{bernstein}, for all $x,u>0$, 
\begin{equation}\label{c1}
\P\cro{|\Pi_{S}\eps|_{2}^{2}\ge \kappa^{2}\pa{\sigma^{2}+{2cu\over \kappa}}\pa{D+x},\ |\Pi_{S}\eps|_{\infty}\le u}\le e^{-x}
\end{equation}
with $\kappa=18$ and 
\begin{equation}\label{c2}
\P\pa{\ab{\Pi_{S}\eps}_{\infty}\ge u}\le 
2n\exp\cro{-{x^{2}\over 2\Lambda_{2}^{2}(S)\pa{\sigma^{2}+cx}}}
\end{equation}
where $\Lambda_{2}(S)$ is defined by~\eref{defL}.
\end{thm}

\begin{proof}
Let us set $\chi=|\Pi_{S}\eps|_{2}$. 
For $t\in S$, let $X_{t}=\<\eps,t\>$ and $t_{0}=0$. It follows from the independence of the $\eps_{i}$ and Inequality~\eref{bernstein} that~\eref{debase} holds with $d(t,s)=\sigma|t-s|_{2}$ and $\delta(t,s)=|t-s|_{\infty}$, for all $s,t\in S$.  
The random variable $\chi$ equals the supremum of the $X_{t}$ when $t$ runs among those elements $t$ of $S$ satisfying $|t|_{2}\le 1$.  Besides, the supremum is achieved for $\hat t=\Pi_{S}\eps/\chi$ and thus, on the event $\ac{\chi\ge z,\ |\Pi_{S}\eps|_{\infty}\le u}$ 
\[
\chi=\sup_{t\in T}X_{t}\ \ {\rm with} \ T=\ac{t\in S,\ |t|_{2}\le 1,\ |t|_{\infty}\le  uz^{-1}}
\]
leading to the bound
\begin{eqnarray*}
\P\pa{\chi\ge z,\ |\Pi_{S}\eps|_{\infty}\le u}\le \P\pa{\sup_{t\in T}X_{t}\ge z}.
\end{eqnarray*}
We take $z=\kappa\sqrt{(\sigma^{2}+2cu\kappa^{-1})(D+x)}$ and (using the concavity of $x\mapsto \sqrt{x}$) note that 
\[
z\ge \kappa\pa{\sqrt{\sigma^{2}(D+x)}+cuz^{-1}(D+x)}.
\]
Then, by applying Theorem~\ref{norm} with $v=\sigma$, $b=cu/z$, we obtain Inequality~\eref{c1}.

Let us now turn to Inequality~\eref{c2}.
Under~\eref{bernstein}, we can apply Bernstein's Inequality~\eref{bern} to $X=\<\eps,t\>$ and $X=\<-\eps,t\>$ with $t\in S$, $v^{2}=\sigma^{2}|t|_{2}^{2}$ and $c|t|_{\infty}$ in place of $c$ and get for all $t\in S$ and $x>0$
\begin{equation}\label{eq1}
\P\pa{|\<\eps,t\>|\ge x}\le 2\exp\cro{-{x^{2}\over 2\pa{\sigma^{2}|t|_{2}^{2}+c|t|_{\infty}x}}}.
\end{equation}
Let us take $t=\Pi_{S}e_{i}$ with $i\in\ac{1,\ldots,n}$. Since $|t|_{2}\le \Lambda_{2}(S)$ and 
\[
|t|_{\infty}=\max_{i,i'=1,\ldots,n}\ab{\<\Pi_{S}e_{i},e_{i'}\>}=\max_{i,i'=1,\ldots,n}\ab{\<\Pi_{S}e_{i},\Pi_{S}e_{i'}\>}\le \Lambda_{2}^{2}(S),
\]
we obtain for all $i\in\ac{1,\ldots,n}$
\begin{eqnarray*}
\P\pa{|\<\Pi_{S}\eps,e_{i}\>|\ge x}&\le& 2\exp\cro{-{x^{2}\over 2\Lambda_{2}^{2}(S)\pa{\sigma^{2}+cx}}}
\end{eqnarray*}
We obtain Inequality~\eref{c2} by summing up these probabilities for $i=1,\ldots,n$.
\end{proof}

\subsection{Proof of Theorem~\ref{selmod}}
Let us fix some $m\in\M$. It follows from simple algebra and the inequality ${\rm crit}(\hat m)\le {\rm crit}(m)$ that 
\[
\ab{f-\hat f_{\hat m}}_{2}^{2}\le \ab{f-\hat f_{m}}_{2}^{2}+2\<\eps,\hat f_{\hat m}-\hat f_{m}\>+\pen(m)-\pen(\hat m).
\]
Using the elementary inequality $2ab\le a^{2}+b^{2}$  for all $a,b\in\R$, we have for  $K>1$,
\begin{eqnarray*}
2\<\eps,\hat f_{\hat m}-\hat f_{m}\>
&\le& 2\ab{\hat f_{\hat m}-\hat f_{m}}_{2}\ab{\Pi_{S_{m}+S_{\hat m}}\eps}_{2}\\
&\le& K^{-1}\ab{\hat f_{\hat m}-\hat f_{m}}_{2}^{2}+K\ab{\Pi_{S_{m}+S_{\hat m}}\eps}_{2}^{2}\\
&\le& K^{-1}\cro{\pa{1+{K-1\over K}}\ab{\hat f_{\hat m}-f}_{2}^{2}+\pa{1+{K\over K-1}}\ab{f-\hat f_{m}}_{2}^{2}}\\
&& \ \ \ +\ \ K\ab{\Pi_{S_{m}+S_{\hat m}}\eps}_{2}^{2},
\end{eqnarray*}
and we derive
\begin{eqnarray*}
{(K-1)^{2}\over K^{2}}\ab{f-\hat f_{\hat m}}_{2}^{2}&\le& {K^{2}+K-1\over K(K-1)}\ab{f-\hat f_{m}}_{2}^{2}+K\ab{\Pi_{S_{m}+S_{\hat m}}\eps}_{2}^{2}-\pa{\pen(\hat m)-\pen(m)}\\
&\le& {K^{2}+K-1\over K(K-1)}\ab{f-\hat f_{m}}_{2}^{2}+\pen(m)\\
&& \ \ + K\ab{\Pi_{S_{m}+S_{\hat m}}\eps}_{2}^{2}-\pa{\pen(\hat m)+\pen(m)}.\\
\end{eqnarray*}
Setting 
\begin{eqnarray*}
A_{1}(\hat m)&=& K\kappa^{2}\pa{\sigma^{2}+{2cu\over \kappa}}\pa{{\ab{\Pi_{S_{m}+S_{\hat m}}\eps}_{2}^{2}\over \kappa^{2}\pa{\sigma^{2}+{2cu\over \kappa}}}-D_{\hat m}-D_{m}-\Delta_{\hat m}-\Delta_{m}}_{+}\1\ac{\ab{\Pi_{S_{m}+S_{\hat m}}\eps}_{\infty}\le u}\\
A_{2}(\hat m)&=& K\ab{\Pi_{S_{m}+S_{\hat m}}\eps}_{2}^{2}\1\ac{\ab{\Pi_{S_{m}+S_{\hat m}}\eps}_{\infty}\ge u}
\end{eqnarray*}
and using~\eref{pen}, we deduce that
\[
{(K-1)^{2}\over K^{2}}\ab{f-\hat f_{\hat m}}_{2}^{2}\le {K^{2}+K-1\over K(K-1)}\ab{f-\hat f_{m}}_{2}^{2}+\pen(m)+A_{1}(\hat m)+A_{2}(\hat m),
\]
and by taking the expectation on both side we get 
\[
{(K-1)^{2}\over K^{2}}\E\cro{\ab{f-\hat f_{\hat m}}_{2}^{2}}\le {K^{2}+K-1\over K(K-1)}\E\cro{\ab{f-\hat f_{m}}_{2}^{2}}+\pen(m)+\E\cro{A_{1}(\hat m)}+\E\cro{A_{2}(\hat m)}.
\]
The index $m$ being arbitrary, it remains to bound  $E_{1}=\E\cro{A_{1}(\hat m)}$ and $E_{2}=\E\cro{A_{2}(\hat m)}$ from above. 

Let $m'$ be some deterministic index in $\M$. By using Theorem~\ref{chi} with $S=S_{m}+S_{m'}$ the dimension of which is not larger than $D_{m}+D_{m'}$ and integrating~\eref{c1} with respect to $x$ we get
\[
\E\cro{A(m')}\le K\kappa^{2}\pa{\sigma^{2}+{2cu\over \kappa}}e^{-\Delta_{m}-\Delta_{m'}}
\] 
and thus 
\[
E_{1}\le \sum_{m'\in\M}\E\cro{A(m')}\le  K\kappa^{2}\pa{\sigma^{2}+{2cu\over \kappa}}\Sigma.
\]

Let us now turn to $\E\cro{A_{2}(\hat m)}$. By using that $S_{\hat m}+S_{m}\subset \S$, $\ab{\Pi_{S_{\hat m}+S_{m}}\xi}_{2}^{2}\le \ab{\Pi_{\S}\xi}_{2}^{2}\le n\ab{\Pi_{\S}\xi}_{\infty}^{2}$. Besides, it follows from the definition of $\overline \Lambda_{\infty}$ that 
\[
\ab{\Pi_{S_{\hat m}+S_{m}}\xi}_{\infty}= \ab{\Pi_{S_{\hat m}+S_{m}}\Pi_{\S}\xi}_{\infty}\le \overline \Lambda_{\infty}\ab{\Pi_{\S}\xi}_{\infty}.
\]
and therefore, setting $x_{0}=\overline \Lambda_{\infty}^{-1}u$
\begin{eqnarray*}
E_{2}&\le& Kn\E\cro{\ab{\Pi_{\S}\xi}_{\infty}^{2}\1\ac{\ab{\Pi_{\S}\xi}_{\infty}\ge x_{0}}}.
\end{eqnarray*}
We shall now use the following lemma the proof of which is deferred to the end of the section.
\begin{lem}\label{lem1}
Let $X$ be some nonnegative random variable satisfying for all $x>0$, 
\begin{equation}\label{bernX}
\P\pa{X\ge x}\le a\exp\cro{-\phi(x)}\ \ \ {\rm with}\ \ \ \phi(x)= {x^{2}\over 2\pa{\alpha+\beta x}}\ \ \ 
\end{equation}
where $a, \alpha>0$ and $\beta\ge 0$. For $x_{0}>0$ such that $\phi(x_{0})\ge 1$,
\[
\E\cro{X^{p}\1\ac{X\ge x_{0}}}\le ax_{0}^{p}e^{-\phi(x_{0})}\pa{1+{ep!\over \phi(x_{0})}},\ \ \ \forall p\ge 1.
\]
\end{lem}

We apply the lemma with $p=2$ and $X=\ab{\Pi_{\S}\xi}_{\infty}$ for which we know from~\eref{c2} that~\eref{bernX} holds with $a=2n$, $\alpha=\Lambda_{2}^{2}(S)\sigma^{2}$ and $\beta=\Lambda_{2}^{2}(S)c$. Besides, it follows from the definition of $x_{0}$ and the fact that $n\ge 2$ that
\[
\phi(x_{0})={x_{0}^{2}\over 2\Lambda_{2}^{2}(S)\pa{\sigma^{2}+cx_{0}}}\ge \log\pa{n^{2}e^{z}}\ge 1. 
\]
The assumptions of Lemma~\ref{lem1} being checked, we deduce that $E_{2}\le 2Kx_{0}^{2}e^{-z}$ and conclude the proof putting these upper bounds on $E_{1}$ and $E_{2}$ together. 

Let us now turn to the proof of the lemma.

\begin{proof}[Proof of Lemma~\ref{lem1}]
Since  
\[
\E\cro{X^{p}\1\ac{X\ge x_{0}}}\le x_{0}^{p}\P\pa{X\ge x_{0}}+\int_{x_{0}}^{+\infty}px^{p-1}\P\pa{X\ge x}dx,
\]
it remains to bound from above the integral. Let us set 
\[
I_{p}=\int_{x_{0}}^{+\infty}px^{p-1}e^{-\phi(x)}dx.
\]
Note that $\phi'$ is increasing and by integrating by parts we have  
\begin{eqnarray*}
I_{p}&=&\int_{x_{0}}^{+\infty}{px^{p-1}\over \phi'(x)}\phi'(x)e^{-\phi(x)}\\
&\le& {p\over \phi'(x_{0})}\cro{x_{0}^{p-1}e^{-\phi(x_{0})}+(p-1)I_{p-1}}.
\end{eqnarray*}
By induction over $p$ and using that $x_{0}\phi'(x_{0})\ge \phi(x_{0})\ge 1$ we get
\begin{eqnarray*}
I_{p}&\le& p!x_{0}^{p}e^{-\phi(x_{0})}\sum_{k=0}^{p-1}{\pa{x_{0}\phi'(x_{0})}^{-(k+1)}\over (p-k-1)!}\le {ep!x_{0}^{p}e^{-\phi(x_{0})}\over \phi(x_{0})}.
\end{eqnarray*}
\end{proof}

\subsection{Proof of Proposition~\ref{histo}}
Let $m$ be some partition of $\ac{1,\ldots, n}$. By applying 
Proposition~\ref{control-lambda} with $J=\ac{1}$, $P=m$ and $\Phi=1$, we obtain 
\[
\Lambda_{2}^{2}(S_{m})\le {1 \over \min_{I\in m}|I|}\ \ {\rm and}\ \ \Lambda_{\infty}(S_{m})\le 1.
\]
In fact, one can check that these inequalities are equalities. Since for all $m\in\M$, $S_{m}\subset S_{\mathfrak{m}}$, we deduce that under~\eref{condhisto}
\[
\Lambda_{2}^{2}(\S)\le \Lambda_{2}^{2}(S_{\mathfrak{m}})\le {1\over a^{2}\log^{2}(n)}
\]
For two partitions $m,m'$ of $\ac{1,\ldots, n}$, define 
\begin{equation}\label{suppart}
m\vee m'=\ac{I\cap I'|\  I\in m,\ I'\in m'}.
\end{equation}
Since the elements of $m,m'$ for $m,m'\in\M$ consist of consecutive integers $S_{m\vee m'}=S_{m}+S_{m'}$ and therefore
\[
\overline{\Lambda}_{\infty}=\sup_{m,m'\in\M}\Lambda_{\infty}(S_{m}+S_{m'})=\sup_{m,m'\in\M}\Lambda_{\infty}(S_{m\vee m'})=1.
\]
The result follows by applying Theorem~\ref{selmod} with $z=b\log(n)$.
 
\subsection{Proof of Proposition~\ref{pp}}
Let  $m$ be a partition of $\ac{1,\ldots,n}$ such that for all 
$I\in m$, $I$ consists of consecutive integers and $|I|> d$. As proved in Mason \& Handscom~\citeyearpar{MR1937591}, an orthonormal basis of $S_{m}$ is given by the vectors $\phi_{j,I}$ defined by 
\[
\<\phi_{0,I},e_{i}\>={1\over \sqrt{|I|}}\1_{I}(i)
\]
and for $j=1,\ldots,d$
\[
\<\phi_{j,I},e_{i}\>=\sqrt{2\over |I|}Q_{j}\pa{\cos\pa{{(i-\min I+1/2)\pi\over |I|}}}\1_{I}(i)
\]
where $Q_{j}$ is the Chebyshev polynomial of degree $j$ defined on $[-1,1]$ by the formula
\[
Q_{j}(x)=\cos(j\theta)\ \ {\rm if}\ \ x=\cos\theta.
\]
By applying Proposition~\ref{control-lambda} with $\Phi=\sqrt{2}$, $P=m$ and $J=\ac{0,\ldots,d}$ and get
\[
\Lambda_{2}^{2}(S_{m})\le {2(d+1)\over \min_{I\in m}|I|}\ \ {\rm and}\ \ \Lambda_{\infty}(S_{m})\le 2(d+1).
\]
Since for those $m\in\M$, $S_{m}\subset S_{\mathfrak{m}}$, $\S=\sum_{m\in \M}S_{m}\subset S_{\mathfrak{m}}$ and therefore
\[
\Lambda_{2}^{2}(\S)\le \Lambda_{2}^{2}(S_{\mathfrak{m}})\le {1\over a^{2}\log^{2}(n)}.
\]
Moreover, since for the elements of $m$ and $m'$ for $m,m'\in\M$ consist of consecutive integers $S_{m}+S_{m'}=S_{m\vee m'}$ with $m\vee m'$ is defined by~\eref{suppart} and 
\[
\sup_{m,m'\in\M}\Lambda_{\infty}(S_{m}+S_{m'})=\sup_{m,m'\in\M}\Lambda_{\infty}(S_{m\vee m'})\le 2(d+1)
\]
which implies that $\overline{\Lambda}_{\infty}\le 2(d+1)$. It remains to apply Theorem~\ref{selmod} with $z=b\log(n)$. 

\subsection{Proof of Proposition~\ref{trigo}}
Let $\mathfrak{m}=\ac{0,\ldots,2\overline D}$. Under the 	assumption that $2\overline D+1\le \sqrt{n}/(a\log(n))$,  for all $m\subset\mathfrak{m}$, the family of vectors $\ac{\phi_{j}}_{j\in m}$ is a orthonormal basis of $S_{m}$. By applying Proposition~\ref{control-lambda} with $P$ reduced to $\ac{\ac{1,\ldots,n}}$, $J=m$, $\Phi=\sqrt{2}$, we get 
\[
\Lambda_{2}^{2}(S_{m})\le {2|m|\over n}\ \ {\rm and}\ \ \Lambda_{\infty}(S_{m})\le \sqrt{n}\Lambda_{2}(S_{m})\le \sqrt{2|m|}.
\]
Since for all $m\in \M$, $S_{m}\subset S_{\mathfrak{m}}$,  $\S=\sum_{m\in\M}S_{m}\subset S_{\mathfrak{m}}$ and therefore
\[
\Lambda_{2}^{2}(\S)\le \Lambda_{2}^{2}(S_{\mathfrak{m}})\le {2(2\overline D+1)\over n}.
\]
Moreover, for all $m,m'\in\M$, $S_{m}+S_{m'}=S_{m\cup m'}$ with $m\cup m'\subset \mathfrak{m}$ and thus, 
\[
\Lambda_{\infty}(S_{m}+S_{m'})\le \sqrt{2(|m\cup m'|}\le\sqrt{2(2\overline D+1)}. 
\]
It remains to apply Theorem~\ref{selmod} with $z=b\log(n)$. 

\thanks{Acknowledgment: We would like to thank Lucien Birg\'e for his helpful comments and for pointing us the book of Talagrand, which has actually been the starting point of this paper.}
\bibliographystyle{apalike}

\end{document}